\documentclass[11pt]{article}

\usepackage{graphicx} %
\usepackage{float}    %
\usepackage{verbatim} %
\usepackage{amsmath}  %
\usepackage{amssymb}  %
\usepackage{subfig}   %
\usepackage[colorlinks=true,citecolor=blue,linkcolor=blue,urlcolor=blue]{hyperref}
\usepackage{fullpage}
\usepackage{amsthm}
\usepackage{enumerate}
\usepackage{paralist}
\usepackage{xspace}
\usepackage{caption}
\usepackage{bbm}
\usepackage{url}
\usepackage{mathrsfs}
\usepackage{cleveref}

\usepackage[sort,square,numbers]{natbib}
\newtheorem{thm}{Theorem}[section]

\newtheorem{mylemma}[thm]{Lemma}

\newtheorem{mythm}[thm]{Theorem}

\DeclareMathOperator*{\Tr}{\mathrm{tr}}

\DeclareMathOperator\erf{erf}

\newcommand{\R}{\ensuremath{\mathbb{R}}}

\newcommand{\N}{\ensuremath{\mathbb{N}}}

\newcommand{\norm}[1]{\lVert #1 \rVert}

\newcommand{\E}{\mathbb{E}}
\newcommand{\abs}[1]{\ensuremath{| #1 |}}

\renewcommand{\Pr}{\mathbb{P}}
\newcommand{\T}{\mathsf{T}}

\newcommand{\calE}{\mathcal{E}}

\newcommand{\calS}{\mathcal{S}}

\newcommand{\bmattwo}[4]{\begin{bmatrix} #1 & #2 \\ #3 & #4 \end{bmatrix}}

\newcommand{\cvectwo}[2]{\begin{bmatrix} #1 \\ #2 \end{bmatrix}}

\numberwithin{equation}{section}

\renewcommand{\geq}{\geqslant}

\renewcommand{\leq}{\leqslant}

\newcommand{\rmd}{\mathrm{d}}

\begin{document}

\title{An elementary proof of anti-concentration
for degree two non-negative Gaussian polynomials} 
\author{Stephen Tu and Ross Boczar}

\maketitle

\begin{abstract}
A classic result by \citet{carbery2001anticonc}
states that a polynomial of Gaussian random 
variables exhibits anti-concentration in the following
sense: for any degree $d$ polynomial $f$, one has the estimate $\Pr\{ \abs{f(x)} \leq \varepsilon \cdot \E\abs{f(x)} \} \leq O(1) \cdot d \varepsilon^{1/d}$, where the probability is over
$x$ drawn from an isotropic Gaussian distribution.
In this note, we give an elementary proof of this
result for the special case when $f$ is a degree two non-negative
polynomial.
\end{abstract}

\section{Introduction}

A well-known result by \citet{carbery2001anticonc} states
that low degree polynomials of Gaussian random variables
exhibit anti-concentration:
\begin{mythm}[{\citet[Theorem~8]{carbery2001anticonc}}]
\label{stmt:carbery_and_wright}
There exists a universal positive constant $C$ such that the following holds.
Let $f : \R^n \rightarrow \R$ be any degree $d$ polynomial,
and let $\mu$ denote any log-concave measure on $\R^n$.
For every $\varepsilon > 0$:
\begin{align*}
    \Pr_{x \sim \mu}\left\{ \abs{f(x)} \leq \varepsilon \cdot \E_{x\sim\mu} \abs{f(x)} \right\} \leq C \cdot d \varepsilon^{1/d}.
\end{align*}
\end{mythm}
Proving \Cref{stmt:carbery_and_wright} is quite non-trivial.
Recently, \citet{lovett2010elementary} provided a more elementary proof of the anti-concentration property,
but the resulting dependence on $d$ is sub-optimal.

The purpose of this note is to supply a proof of \Cref{stmt:carbery_and_wright}, with reasonably sharp constants, 
in the special case when $\mu$ is the isotropic
Gaussian distribution in $\R^n$, $d=2$, and $f$ is
non-negative. Since every degree two non-negative polynomial
can be written as a simple quadratic form, the proof
only contains rudimentary calculations.
The explicit result we will show is stated below:
\begin{mythm}
\label{stmt:special_case}
Let $f : \R^n \rightarrow \R$ be any non-negative degree two polynomial. For every $\varepsilon > 0$:
\begin{align*}
    \Pr_{x \sim N(0, I_n)}\left\{ f(x) \leq \varepsilon \cdot \E_{x \sim N(0, I_n)}[f(x)] \right\} \leq (2e \cdot \varepsilon)^{1/2}.
\end{align*}
\end{mythm}

\section{Proof of \Cref{stmt:special_case}}
In general, a non-negative polynomial 
does not necessarily have a sum-of-squares decomposition.
However, in the degree $d=2$ case, 
they are equivalent representations:
every non-negative degree-two polynomial $f(x)$ can be written as
the following quadratic form:
\begin{align*}
    f(x) = \cvectwo{1}{x}^\T \bmattwo{q_{11}}{q_{12}^\T}{q_{12}}{Q_{22}} \cvectwo{1}{x}, \quad \bmattwo{q_{11}}{q_{12}^\T}{q_{12}}{Q_{22}} \in \calS^{n+1}_{\geq 0}.
\end{align*}
Here, $\calS^n_{\geq 0}$ (resp.\ $\calS^{n}_{> 0}$)
denote the cone of real-valued
$n \times n$ symmetric positive semidefinite (resp.\ positive definite)
matrices.
With this notation, we have that
$\E_{x \sim N(0, I_n)}[f(x)] = q_{11} + \Tr(Q_{22})$.
Thus, \Cref{stmt:special_case} is equivalent to showing
for all $Q \in \calS^{n+1}_{\geq 0}$:
\begin{align}
    \Pr\left\{ \cvectwo{1}{x}^\T Q \cvectwo{1}{x} \leq \varepsilon \cdot \Tr(Q) \right\} \leq (2e \cdot \varepsilon)^{1/2} \quad \forall \varepsilon > 0. \label{eq:equivalent_to_show}
\end{align}

The first step towards showing \eqref{eq:equivalent_to_show} is the following
upper bound on the moment generating function of the quadratic form:
\begin{mylemma}
\label{stmt:quadratic_gaussian_mgf}
Let $x \sim N(0, I_n)$ 
and let $Q = \bmattwo{q_{11}}{q_{12}^\T}{q_{12}}{Q_{22}} \in \calS^{n+1}_{\geq 0}$.
For any $\lambda > 0$, we have:
\begin{align*}
    \E \exp\left(-\lambda \cvectwo{1}{x}^\T  \bmattwo{q_{11}}{q_{12}^\T}{q_{12}}{Q_{22}} \cvectwo{1}{x}\right) \leq \det(I_n + 2 \lambda Q_{22})^{-1/2}.
\end{align*}
\end{mylemma}
\begin{proof}
Define $\mu := -(Q_{22} + (2\lambda)^{-1} I_n)^{-1/2} q_{12}$. We have:
\begin{align*}
   &\E \exp\left(-\lambda \cvectwo{1}{x}^\T  \bmattwo{q_{11}}{q_{12}^\T}{q_{12}}{Q_{22}} \cvectwo{1}{x}\right) \\
   &= (2\pi)^{-n/2} \int \exp\left\{ -\lambda \left[ q_{11} + 2 q_{12}^\T x  +  x^\T [Q_{22} + (2\lambda)^{-1} I_n] x  + \mu^\T \mu - \mu^\T \mu \right] \right\} \rmd x \\ 
   &= (2\pi)^{-n/2} \int \exp\left\{ -\lambda \left[ q_{11} -\norm{\mu}_2^2 + \norm{(Q_{22} + (2\lambda)^{-1} I_n)^{1/2} x - \mu}_2^2 \right] \right\} \rmd x \\
   &= (2\pi)^{-n/2} \exp\left\{ -\lambda \left[ q_{11} - q_{12}^\T (Q_{22} + (2\lambda)^{-1} I_n)^{-1} q_{12} \right] \right\} \\
   &\qquad \times \int \exp\left\{ -\lambda \norm{(Q_{22} + (2\lambda)^{-1} I_n)^{1/2} x - \mu}_2^2  \right\} \rmd x \\
   &= \det(I_n + 2\lambda Q_{22})^{-1/2}  \exp\left\{ -\lambda \left[ q_{11} - q_{12}^\T (Q_{22} + (2\lambda)^{-1} I_n)^{-1} q_{12} \right] \right\}.
\end{align*}
Next, observe that the quantity
$ q_{11} - q_{12}^\T (Q_{22} + (2\lambda)^{-1} I_n)^{-1} q_{12}$ is non-negative, since it 
is the Schur complement of the positive semidefinite matrix:
\begin{align*}
    \bmattwo{q_{11}}{q_{12}^\T}{q_{12}}{Q_{22}} + \bmattwo{0}{0}{0}{(2\lambda)^{-1} I_n}
\end{align*}
Therefore:
\begin{align*}
     \exp\left\{ -\lambda  [q_{11} - q_{12}^\T (Q_{22} + (2\lambda)^{-1} I_n)^{-1} q_{12} ]\right\} \leq 1.
\end{align*}
The claim now follows.
\end{proof}

The moment generating function bound from 
\Cref{stmt:quadratic_gaussian_mgf} is sufficient to obtain
the following weaker form of \eqref{eq:equivalent_to_show}
via Chernoff's inequality:
\begin{mylemma}
\label{stmt:small_ball_gaussian_quadratic_explicit}
Let $x \sim N(0, I_n)$ 
and let $Q = \bmattwo{q_{11}}{q_{12}^\T}{q_{12}}{Q_{22}} \in \calS^{n+1}_{\geq 0}$.
For any $\varepsilon > 0$,
\begin{align*}
    \Pr\left\{  \cvectwo{1}{x}^\T  \bmattwo{q_{11}}{q_{12}^\T}{q_{12}}{Q_{22}} \cvectwo{1}{x} \leq \varepsilon \cdot \Tr(Q_{22}) \right\} \leq (e \cdot \varepsilon)^{1/2}.
\end{align*}
\end{mylemma}
\begin{proof}
Let $\{\lambda_i\}_{i=1}^{n}$ denote the eigenvalues of $Q_{22}$.
We observe that for any $\eta > 0$, since all the eigenvalues of $Q_{22}$ are
non-negative:
\begin{align*}
    \det(I_n + 2\eta Q_{22}) = \prod_{i=1}^{n} (1 + 2 \eta \lambda_i) \geq 1 + 2 \eta \sum_{i=1}^{n} \lambda_i = 1 + 2 \eta \Tr(Q_{22}).
\end{align*}
Now by a Chernoff bound:
\begin{align*}
     &\Pr\left\{  \cvectwo{1}{x}^\T  \bmattwo{q_{11}}{q_{12}^\T}{q_{12}}{Q_{22}} \cvectwo{1}{x} \leq \varepsilon \cdot \Tr(Q_{22}) \right\} \\
     &\leq \inf_{\eta > 0} \exp\{\eta \varepsilon \Tr(Q_{22})\}  \E \exp\left(-\eta \cvectwo{1}{x}^\T  \bmattwo{q_{11}}{q_{12}^\T}{q_{12}}{Q_{22}} \cvectwo{1}{x}\right) \\
     &\leq \inf_{\eta > 0}  \exp\{\eta \varepsilon \Tr(Q_{22})\} \det(I_n + 2\eta Q_{22})^{-1/2} &&\text{using \Cref{stmt:quadratic_gaussian_mgf}} \\
     &\leq \inf_{\eta > 0} \exp\{\eta \varepsilon \Tr(Q_{22})\} (1 + 2 \eta \Tr(Q_{22}))^{-1/2} \\
     &\leq (e^{1-\varepsilon} \cdot \varepsilon)^{1/2} &&\text{setting } \eta = (1-\varepsilon)/(2 \varepsilon \Tr(Q_{22})) \\
     &\leq (e \cdot \varepsilon)^{1/2}.
\end{align*}
\end{proof}

At this point, to complete the proof of \Cref{stmt:special_case},
we need the following small-ball probability estimate for a non-centered
Gaussian random variable:
\begin{mylemma}
\label{stmt:mean_anticoncentration}
Let $x \sim N(\mu, \sigma^2)$ with $\mu \in \R$ and $\sigma > 0$.
For any $\varepsilon \in (0,1)$, we have:
\begin{align*}
    \Pr\{ \abs{x} \leq \varepsilon \abs{\mu} \} \leq \varepsilon/2.
\end{align*}
\end{mylemma}
\begin{proof}
We have:
\begin{align*}
    \Pr\{\abs{x} \leq \varepsilon\abs{\mu}\} &= \Pr_{g\sim N(0, 1)}\{ \abs{\mu + \sigma g} \leq \varepsilon \abs{\mu}\} \\
    &= \frac{1}{\sqrt{2\pi}} \int_{-(1+\varepsilon)\abs{\mu}/\sigma}^{-(1-\varepsilon)\abs{\mu}/\sigma} \exp(-x^2/2) \,\rmd x \\
    &\leq \sup_{\sigma > 0} \sup_{r > 0}  h(r,\sigma;\varepsilon) := \frac{1}{\sqrt{2\pi}}\int_{-(1+\varepsilon)r/\sigma}^{-(1-\varepsilon)r/\sigma} \exp(-x^2/2) \,\rmd x .
\end{align*}
Letting $\erf{x} = \frac{2}{\sqrt{\pi}} \int_0^x e^{-t^2} \,\rmd t$
denote the Gaussian error function, we have:
\begin{align*}
    h(r,\sigma;\varepsilon) = \frac{1}{2}\left[ \erf\left( \frac{-(1-\varepsilon)r}{\sqrt{2}\sigma}\right) + \erf\left( \frac{(1+\varepsilon)r}{\sqrt{2}\sigma}\right) \right].
\end{align*}
Computing $\frac{\partial h(r,\sigma;\varepsilon)}{\partial r} = 0$
yields the unique real-valued root:
\begin{align*}
    r_\star = \sigma \sqrt{ \frac{\log\left(\frac{1+\varepsilon}{1-\varepsilon}\right)}{2\varepsilon}}.
\end{align*}
Evaluating $\frac{\partial^2 h(r,\sigma;\varepsilon)}{\partial^2 r}\big|_{r=r_\star}$ yields:
\begin{align*}
    \frac{\partial^2 h(r,\sigma;\varepsilon)}{\partial^2 r}\bigg|_{r=r_\star} = -\frac{2 (1-\varepsilon) \varepsilon \left(\frac{1+\varepsilon}{1-\varepsilon}\right)^{-\frac{(1-\varepsilon)^2}{4 \varepsilon}} \sqrt{\frac{\log
   \left(\frac{1+\varepsilon}{1-\varepsilon}\right)}{\varepsilon}}}{\sqrt{\pi } \sigma^2} < 0,
\end{align*}
and hence for fixed $\sigma, \varepsilon$, the quantity $r_\star$ maximizes
$r \mapsto h(r,\sigma;\varepsilon)$ by the second derivative test:
\begin{align*}
    \sup_{\sigma > 0} \sup_{r > 0} h(r,\sigma;\varepsilon) &= \sup_{\sigma > 0} \frac{1}{2}\left[ \erf\left(\frac{-(1-\varepsilon)}{\sqrt{2}} \sqrt{ \frac{\log\left(\frac{1+\varepsilon}{1-\varepsilon}\right)}{2\varepsilon}} \right) + \erf\left(\frac{1+\varepsilon}{\sqrt{2}} \sqrt{ \frac{\log\left(\frac{1+\varepsilon}{1-\varepsilon}\right)}{2\varepsilon}} \right) \right] \\
    &= \frac{1}{2}\left[ \erf\left(\frac{-(1-\varepsilon)}{\sqrt{2}} \sqrt{ \frac{\log\left(\frac{1+\varepsilon}{1-\varepsilon}\right)}{2\varepsilon}} \right) + \erf\left(\frac{1+\varepsilon}{\sqrt{2}} \sqrt{ \frac{\log\left(\frac{1+\varepsilon}{1-\varepsilon}\right)}{2\varepsilon}} \right) \right] \\
    &=: \zeta(\varepsilon).
\end{align*}
One can check that $\lim_{\varepsilon \rightarrow 0^{+}} \zeta(\varepsilon) = 0$
and $\lim_{\varepsilon \rightarrow 1^{-}} \zeta(\varepsilon) = 1/2$.
Furthermore:
\begin{align*}
    \zeta''(\varepsilon) = \frac{\left(\frac{1+\varepsilon}{1-\varepsilon}\right)^{-\frac{(1-\varepsilon)^2}{4 \varepsilon}} \left(\left(\varepsilon^2-1\right) \log
   \left(\frac{1+\varepsilon}{1-\varepsilon}\right)+2 \varepsilon\right)^2}{4 \sqrt{\pi } (1-\varepsilon) \varepsilon^3 (1+\varepsilon)^2
   \sqrt{\frac{\log \left(\frac{1+\varepsilon}{1-\varepsilon}\right)}{\varepsilon}}} > 0.
\end{align*}
Hence, the function $\zeta(\varepsilon)$ is also convex
on $(0, 1)$. Therefore, $\zeta(\varepsilon) \leq \varepsilon / 2$, from which the claim follows.
\end{proof}

With both \Cref{stmt:quadratic_gaussian_mgf}
and \Cref{stmt:mean_anticoncentration} in hand, we are now ready
to prove \eqref{eq:equivalent_to_show}, from which \Cref{stmt:special_case}
immediately follows.
\begin{mylemma}
Let $x \sim N(0, I_n)$ 
and let $Q = \bmattwo{q_{11}}{q_{12}^\T}{q_{12}}{Q_{22}} \in \calS^{n+1}_{\geq 0}$. For any $\varepsilon > 0$,
\begin{align*}
    \Pr\left\{  \cvectwo{1}{x}^\T  \bmattwo{q_{11}}{q_{12}^\T}{q_{12}}{Q_{22}} \cvectwo{1}{x} \leq \varepsilon \cdot (q_{11} + \Tr(Q_{22})) \right\} \leq (2e \cdot \varepsilon)^{1/2}.
\end{align*}
\end{mylemma}
\begin{proof}
We can assume wlog that $\varepsilon \in (0, 1/(2e))$, otherwise
there is nothing to prove.
We first suppose that $q_{11} \leq \Tr(Q_{22})$. Then,
by \Cref{stmt:small_ball_gaussian_quadratic_explicit}:
\begin{align*}
    \Pr\left\{  \cvectwo{1}{x}^\T  \bmattwo{q_{11}}{q_{12}^\T}{q_{12}}{Q_{22}}  \cvectwo{1}{x} \leq \varepsilon \cdot (q_{11} + \Tr(Q_{22})) \right\} 
    &\leq \Pr\left\{  \cvectwo{1}{x}^\T  \bmattwo{q_{11}}{q_{12}^\T}{q_{12}}{Q_{22}}  \cvectwo{1}{x} \leq 2\varepsilon \cdot \Tr(Q_{22}) \right\} \\
    &\leq (2e \cdot \varepsilon)^{1/2}.
\end{align*}
Now we assume that $q_{11} > \Tr(Q_{22})$
for the remainder of the proof.
We have:
\begin{align*}
     \Pr\left\{  \cvectwo{1}{x}^\T  \bmattwo{q_{11}}{q_{12}^\T}{q_{12}}{Q_{22}}  \cvectwo{1}{x} \leq \varepsilon \cdot (q_{11} + \Tr(Q_{22})) \right\} 
     \leq \Pr\left\{  \cvectwo{1}{x}^\T  \bmattwo{q_{11}}{q_{12}^\T}{q_{12}}{Q_{22}}  \cvectwo{1}{x} \leq 2\varepsilon \cdot q_{11} \right\}.
\end{align*}
We proceed with a limiting argument. Fix any $\gamma > 0$.
By completing the square:
\begin{align*}
    &q_{11}+\gamma + 2 q_{12}^\T x + x^\T Q_{22} x \\
    &= ((q_{11}+\gamma)^{1/2} + (q_{11}+\gamma)^{-1/2}q_{12}^\T x)^2 + x^\T (Q_{22} - q_{12} (q_{11} + \gamma)^{-1} q_{12}^\T) x \\
    &\geq ((q_{11}+\gamma)^{1/2} + (q_{11}+\gamma)^{-1/2}q_{12}^\T x)^2,
\end{align*}
where the last inequality holds since
$Q_{22} - q_{12} (q_{11}+\gamma )^{-1} q_{12}^\T$
is positive semidefinite because it 
is the Schur complement of the positive semidefinite matrix:
\begin{align*}
    \bmattwo{q_{11}}{q_{12}^\T}{q_{12}}{Q_{22}} + \bmattwo{\gamma}{0}{0}{0}.
\end{align*}
Therefore:
\begin{align*}
    &\Pr\left\{  \cvectwo{1}{x}^\T  \bmattwo{q_{11}}{q_{12}^\T}{q_{12}}{Q_{22}}  \cvectwo{1}{x} \leq 2\varepsilon \cdot q_{11} - (1-2\varepsilon)\gamma  \right\} \\
    &= \Pr\left\{  \cvectwo{1}{x}^\T  \bmattwo{q_{11}+\gamma }{q_{12}^\T}{q_{12}}{Q_{22}} \cvectwo{1}{x} \leq 2\varepsilon \cdot (q_{11}+\gamma )  \right\} \\
    &\leq \Pr\left\{ ((q_{11}+\gamma)^{1/2} + (q_{11}+\gamma)^{-1/2}q_{12}^\T x)^2 \leq 2\varepsilon \cdot (q_{11} + \gamma) \right\} \\
    &= \Pr\left\{ \abs{ (q_{11}+\gamma)^{1/2} + (q_{11}+\gamma)^{-1/2} q_{12}^\T x}  \leq \sqrt{2\varepsilon} \cdot (q_{11}+\gamma)^{1/2} \right\} \\
    &\leq \sqrt{2\varepsilon}/2 &&\text{using \Cref{stmt:mean_anticoncentration}, since $2\varepsilon < 1$}.
\end{align*}
Let $\{\gamma_k\}_{k \geq 1}$ be any positive sequence which
is monotonically decreasing and satisfies $\lim_{k \rightarrow \infty} \gamma_k = 0$.
Define $\calE_k$ as the event:
\begin{align*}
    \calE_k := \left\{ \cvectwo{1}{x}^\T  \bmattwo{q_{11}}{q_{12}^\T}{q_{12}}{Q_{22}}  \cvectwo{1}{x} \leq 2\varepsilon \cdot q_{11} - (1-2\varepsilon)\gamma_k  \right\}.
\end{align*}
Since $\varepsilon \in (0, 1/(2e))$,
we have $1-2\varepsilon > 0$,
and hence $\calE_k \subseteq \calE_{k+1}$ for all $k \in \N_{+}$.
By continuity of measure from below:
\begin{align*}
    \Pr\left\{ \cvectwo{1}{x}^\T  \bmattwo{q_{11}}{q_{12}^\T}{q_{12}}{Q_{22}}  \cvectwo{1}{x} \leq 2\varepsilon \cdot q_{11} \right\} = \Pr\left(\bigcup_{k \geq 1} \calE_k\right) = \lim_{k \rightarrow \infty} \Pr(\calE_k) \leq \sqrt{2\varepsilon}/2.
\end{align*}
The claim now follows.
\end{proof}

\section*{Acknowledgements}
We thank Benjamin Recht for suggesting the proof strategy of
\Cref{stmt:mean_anticoncentration}.

\bibliography{paper}

\end{document}